%% file: pin-seq-gr.tex
\documentclass[final, nomarks]{dmtcs-episciences}

\usepackage{amsmath}
\usepackage{amssymb}%,latexsym}
\usepackage{mathabx}
\usepackage{amsthm}
\usepackage{tikz}

\usepackage[utf8]{inputenc}
\usepackage{subfigure}
\usepackage[round]{natbib}

\input{tikzperm}

\newtheorem{thm}{Theorem}[section]
\newtheorem{prop}[thm]{Proposition}
\newtheorem{cor}[thm]{Corollary}
\newtheorem{lemma}[thm]{Lemma}
\newtheorem{obs}[thm]{Observation}
\newcommand{\AAA}{\mathcal{A}}

\newcommand{\CCC}{\mathcal{C}}
\let\C\CCC
\newcommand{\DDD}{\mathcal{D}}

\newcommand{\LLL}{\mathcal{L}}
\newcommand{\MMM}{\mathcal{M}}

\newcommand{\XXX}{\mathcal{X}}

\newcommand{\gr}{\mathrm{gr}}
\newcommand{\ugr}{\overline{\gr}}
\newcommand{\lgr}{\underline{\gr}}

\newcommand{\gridded}{\sharp}
\newcommand{\sfl}{\mathsf{l}}
\newcommand{\sfr}{\mathsf{r}}
\newcommand{\sfu}{\mathsf{u}}
\newcommand{\sfd}{\mathsf{d}}
\newcommand{\sflu}{\sfl_\sfu}
\newcommand{\sfru}{\sfr_\sfu}
\newcommand{\sfld}{\sfl_\sfd}
\newcommand{\sfrd}{\sfr_\sfd}
\newcommand{\sful}{\sfu_\sfl}
\newcommand{\sfdl}{\sfd_\sfl}
\newcommand{\sfur}{\sfu_\sfr}
\newcommand{\sfdr}{\sfd_\sfr}

\tikzstyle{vertex}=[circle, draw, fill=black,
                        inner sep=0pt, minimum width=4pt]

\newcommand{\smallgrid}{\setplotptradius{2pt}\draw (-.5,0)--++(1,0) (0,-.5)--++(0,1);}
\newcommand{\sept}{\tikz[scale=.5,baseline=-3pt]{\smallgrid\node[permpt] at (.25,-.25) {};}}
\newcommand{\swpt}{\tikz[scale=.5,baseline=-3pt]{\smallgrid\node[permpt] at (-.25,-.25) {};}}
\newcommand{\nept}{\tikz[scale=.5,baseline=-3pt]{\smallgrid\node[permpt] at (.25,.25) {};}}
\newcommand{\nwpt}{\tikz[scale=.5,baseline=-3pt]{\smallgrid\node[permpt] at (-.25,.25) {};}}
\newcommand{\netwo}{\tikz[scale=.5,baseline=-3pt]{\smallgrid\node[permpt] at (.15,.35) {};\node[permpt] at (.35,.15) {};}}
\newcommand{\nwtwo}{\tikz[scale=.5,baseline=-3pt]{\smallgrid\node[permpt] at (-.15,.35) {};\node[permpt] at (-.35,.15) {};}}
\newcommand{\swtwo}{\tikz[scale=.5,baseline=-3pt]{\smallgrid\node[permpt] at (-.15,-.35) {};\node[permpt] at (-.35,-.15) {};}}
\newcommand{\nthree}{\tikz[scale=.5,baseline=-3pt]{\smallgrid\node[permpt] at (.15,.35) {};\node[permpt] at (.35,.15) {};\node[permpt] at (-.25,.25) {};}}
\newcommand{\wthree}{\tikz[scale=.5,baseline=-3pt]{\smallgrid\node[permpt] at (-.15,-.35) {};\node[permpt] at (-.35,-.15) {};\node[permpt] at (-.25,.25) {};}}
\newcommand{\nwfive}{\tikz[scale=.5,baseline=-3pt]{\smallgrid\node[permpt] at (.15,.35) {};\node[permpt] at (.35,.15) {};\node[permpt] at (-.15,-.35) {};\node[permpt] at (-.35,-.15) {};\node[permpt] at (-.25,.25) {};}}

\newcommand{\plottwocellperm}[2][black]{ %[colour]{permutation}
    \foreach \i [count=\nn] in {#2} {\global\let\n\nn};
	\foreach \j [count=\i] in {#2} { %The number 0 draws a vertical bar
        \ifnum0=\j {%
          \draw (\i,.5) -- ++(0,\n-1);   
        } \else {%
 		  \node[permpt,fill=#1,draw=#1] (\j) at (\i,\j) {};
	    } \fi
	};
}

\newcommand{\twocell}[1]{\tikz[scale=.4,baseline=-3pt]{\setplotptradius{2pt}
\foreach \i [count=\nn] in {#1} {\global\let\n\nn};
\begin{scope}[yscale=1/(\n-1),xscale=1/3]
	\foreach \j [count=\i] in {#1} { %The number 0 draws a vertical bar
        \ifnum0=\j {%
          \draw (\i,.5-\n/2) -- ++(0,\n-1);   
        } \else {%
 		  \node[permpt] (\j) at (\i,\j-\n/2) {};
	    } \fi
	};
\end{scope}
}}

\title{Pin classes II: Small pin classes}

\author[Robert Brignall and Ben Jarvis]{Robert Brignall \and Ben Jarvis\thanks{Supported by the Engineering and Physical Sciences Research Council, UK [EP/V520147/1]}}
\affiliation{School of Mathematics and Statistics, The Open University,  Milton Keynes, UK}

\begin{document}
\publicationdata{vol. 27:1, Permutation Patterns 2024}{2026}{11}{10.46298/dmtcs.14896}{2024-12-05; 2024-12-05; 2025-10-09}{2026-03-26}

\maketitle

\begin{abstract}
Pin permutations play an important role in the structural study of permutation classes, most notably in relation to simple permutations and well-quasi-ordering, and in enumerative consequences arising from these. In this paper, we continue our study of \emph{pin classes}, which are permutation classes that comprise all the finite subpermutations contained in an infinite pin permutation. We show that there is a phase transition  at $\mu\approx 3.28277$: there are uncountably many different pin classes whose growth rate is equal to $\mu$, yet only countably many below $\mu$. Furthermore, by showing that all pin classes with growth rate less than $\mu$ are essentially defined by pin permutations that possess a periodic structure, we classify the set of growth rates of pin classes up to $\mu$.
\end{abstract}

%
%
%
%
%
%
%
%
%%%%%%%%%%%%%%%%%%%%%%
\section{Introduction}

This paper continues the study of pin classes started in~\cite{jarvis:pin-classes-i}. Here, we investigate the spectrum of `small' growth rates of such classes, up to growth rate $\mu\approx 3.28277$. This work follows the line of enquiry undertaken by several papers -- for example~\cite{Huczynska2006,kaiser:on-growth-rates:} and \cite{vatter:every-growth-rate:,vatter:small-permutati:,vatter:growth-rates-of:} -- which explored the spectrum of growth rates in general permutation classes. Of special note are~\cite{pantone:growth-rates:} who established a complete characterisation of the real numbers up to $\xi\approx 2.30522$ that are permutation class growth rates, and~\cite{bevan:intervals} who showed that every real number greater than $\lambda\approx 2.35526$ is the growth rate of some permutation class. 

Many of the features of the spectrum of growth rates up to $\xi$ and from $\lambda$ are driven by a single underpinning object: the \emph{oscillations}. These objects govern several phase transitions in the possible growth rates and number of distinct classes, most notably at $\kappa\approx 2.20557$ (the largest real root of $z^3-2z^2-1$) which represents the smallest growth rate for which there exist uncountably many distinct permutation classes -- see~\cite{vatter:small-permutati:}. %A complete classification to the most general question (which would involve studying objects in the gap between $\xi$ and $\lambda$) is temptingly close, and would likely necessitate an improved understanding of the permutations that can be derived from long oscillations.

Oscillations are one type of \emph{pin permutation} (for full definitions, see Section~\ref{sec-preliminaries}), which were originally introduced by~\citet*{brignall:decomposing-sim:} as a tool in a Ramsey-theoretic study of simple permutations. Since then, they have been studied from a number of viewpoints:~\citet*{bassino:a-polynomial-al:b} describe an efficient decision procedure to determine whether a given permutation class contains arbitrarily long pin permutations, while~\citet*{bassino:enumeration-of-:} showed that the permutation class formed of all pin permutations and their subpermutations has a rational generating function. More recently,~\cite{bv:wqo-uncountable:} used pin permutations to construct uncountably many well-quasi-ordered permutation classes with distinct enumerations, thereby disproving a conjecture of~\cite{vatter:survey}. Finally,~\citet*{chudnovsky:unavoidable-induced:} use the term \emph{chain} to describe an analogous object in the study of hereditary classes of graphs.

Given an (infinite) pin permutation, the \emph{pin class} is formed by taking all finite subpermutations that are contained in it.  In the previous paper by~\cite{jarvis:pin-classes-i}, it was shown (among other results) that pin classes always have growth rates. In this paper, we consider the pin classes with smallest growth rate, and establish a phase transition at growth rate $\mu  \approx 3.28277$, the largest real zero of $z^5-4z^4+3z^3-2z^2-z+2$. Specifically:

\paragraph{Theorem~\ref{thm-at-mu}} \emph{There are uncountably many pin classes $\C$ with growth rate equal to $\mu$.}
	
\paragraph{Theorem~\ref{thm-classification}} \emph{Any pin class $\C$ for which $\gr(\C)<\mu$ is defined by a pin permutation whose structure is ultimately periodic, and has growth rate equal either to $\kappa$ or to $\nu_{\ell}$ for some $\ell \geq 1$, defined to be the largest real zero of 
	\[
		z^{2\ell}-4z^{2\ell-1}+3z^{2\ell-2}-2z^{2\ell-3}-z^{2\ell-4}+2z^{2\ell-5}+1.
	\]}

It has been known for some time that $\kappa$ is the least growth rate of a pin class (thanks to the appearance of infinite oscillations established by~\cite{vatter:small-permutati:}). One interesting observation arising from our work is that the second least growth rate of a pin class is not reached until $\nu=\nu_{1}\approx 3.06918$. Also of note is that, among the uncountable collection of pin classes identified that have growth rate $\mu\approx 3.28277$, there exist classes whose defining pin sequence is not eventually recurrent (Proposition~\ref{prop-liouville-V}), even though all pin classes with smaller growth rates must be defined by a periodic or eventually periodic sequence.

The rest of this paper is organised as follows. In Section~\ref{sec-preliminaries}, we cover some basic definitions and results, including those concerning growth rates, the `box sum' decomposition, pin sequences and pin permutations, and the combinatorics of words. Section~\ref{sec-pin-class-structure} develops some theory for pin permutations and begins the process of restricting which pin permutations need to be considered when working below growth rate $\mu$. Sections~\ref{sec-two-quadrants} and~\ref{sec-classification} together fully classify the pin classes whose growth rate is less than $\mu$, while the intervening Section~\ref{sec-non-periodic} exhibits various pin classes that occur at growth rate $\mu$.

%
%
%
%
%
%%%%%%%%%%%%%%%%%%
\section{Preliminaries}\label{sec-preliminaries}

We refer the reader to~\cite{bevan2015defs} for background information and basic definitions concerning permutations and permutation classes. 

The following definition is derived from the study of permutation classes, but here note that we present it in fuller generality. Let $(s_n)$ be a sequence of non-negative integers. The \emph{upper} and \emph{lower growth rates} are defined, respectively, to be 
\[
\ugr((s_n)) = \limsup_{n\to\infty}\sqrt[n]{s_n}, \quad\text{ and }\quad \lgr((s_n))=\liminf_{n\to\infty}\sqrt[n]{s_n},
\]
when these quantities exist. If $\ugr((s_n))=\lgr((s_n))$, then the sequence has an actual \emph{growth rate}, typically denoted $\gr((s_n))$.

We will frequently abuse this notation, for example if we are given some collection $\C$ of combinatorial objects, then $\ugr(\C)$ and $\lgr(\C)$ are defined using the sequence $(|\C_n|)$ that counts the number of objects in $\C$ of each size $n$. If $\C$ is a proper permutation class, for example, then this recovers the standard notions of growth rates for permutation classes. Similarly, if we are given a generating function $f(z)$, then we can also use the term $\gr(f(z))$ to refer to the growth rate of the sequence of coefficients in the Taylor expansion of $f(z)$, or, equivalently (by standard analytic combinatorics), the reciprocal of the radius of convergence of the formal power series $f(z)$.

Recall that the \emph{plot} of a permutation $\pi$ is the points $(i,\pi(i))$ in the Euclidean plane. Consider a set $\LLL$ of horizontal and vertical lines drawn in this plane -- we refer to this as a \emph{gridded plane}. For a permutation $\pi$ and a collection $\LLL$ of lines defining a gridded plane, the \emph{gridded permutation} $\pi^\gridded$ is the pair $(\pi,\LLL)$. Such an object is typically regarded as a division of the plot of $\pi$ into a grid of rectangular cells.

For enumerative purposes, the length of a gridded permutation is the same as its underlying ungridded version. However, two griddings $\pi^\gridded$ and $\pi^\natural$ of the same underlying permutation $\pi$ are only regarded as equal if they have the same number of vertical and horizontal lines, and if each rectangular cell of one gridding contains exactly the same points of $\pi$ as the corresponding rectangular cell of the other. 

Given a permutation class $\C$ and fixed non-negative integers $k,\ell$, we then let
\[\C^\gridded = \{(\pi,\LLL)\ :\ \pi\in\C\text{ and }\LLL\text{ has exactly }k\text{ horizontal and }\ell\text{ vertical lines}\},\] 
that is, the collection of all permutations in $\C$, each divided in every possible way using the prescribed number of lines. Clearly every permutation in $\C$ gives rise to multiple elements of $\C^\gridded$, but in the context of growth rates this presents no problem, as the following result shows. 

\begin{lemma}[C.f.~\cite{vatter:small-permutati:} Proposition 2.1]\label{lem-gr-gridded}
	Let $\C$ be a permutation class, and let $\C^\gridded$ denote the set of gridded permutations with some fixed numbers $k$ of horizontal and $\ell$ of vertical lines.  
	Then $\ugr(\C) = \ugr(\C^\gridded)$ and $\lgr(\C) = \lgr(\C^\gridded)$. 
\end{lemma}

Since the wording of the above lemma differs somewhat from its other appearances in the literature, we provide the short proof. 

\begin{proof}
For a permutation of length $n$, there are precisely $\binom{n+k}{k}$ distinct ways to place the $k$ horizontal lines, and $\binom{n+\ell}{\ell}$ ways to place the vertical ones. Thus, each $\pi\in\CCC$ gives rise to precisely $P(n) = \binom{n+k}{k}\binom{n+\ell}{\ell}$ different gridded permutations in $\C^\gridded$, and so
\[
|\C^\gridded_n| = P(n)|\C_n|,
\]
where $\C_n$ denotes the set of permutations of length $n$ in $\C$. The lemma now follows by noting that the function $P(n)$ is a polynomial of degree $k+\ell$, and hence $\sqrt[n]{P(n)}\to 1$ as $n\to\infty$.
\end{proof}

The above lemma guarantees that, as far as  asymptotic enumeration is concerned, we can move freely between gridded and ungridded permutations. We would like to work \emph{exclusively} with gridded permutations, and for this we need to extend the standard notion of permutation containment (see Bevan~\cite{bevan2015defs} for a formal definition): given two gridded permutations $\sigma^\gridded$ and $\pi^\gridded$ which possess the same number of horizontal and vertical lines, we say that $\sigma^\gridded\leq\pi^\gridded$ if there is an injection $f$ that embeds $\sigma$ into $\pi$ so that $\sigma^\gridded$ is equal (as a gridded permutation) to the image $f(\sigma)$ in the gridded plane containing $\pi$.

While we have defined the set $\C^\gridded$ to comprise \emph{all} possible griddings of permutations in $\C$, we are typically only interested in sets made up of particular griddings: providing every $\pi\in\C$ possesses at least one gridded version in this set, then the upper and lower growth rates of the set will again coincide with those of $\C$.

\paragraph{Pin permutations}
The \emph{rectangular hull} of a set of points in the plot of a permutation (or gridded permutation) is the smallest axis-parallel rectangle that contains them. Let $(p_1,\dots,p_k)$ be a sequence of points in the plane (or in a gridded plane). A \emph{proper pin} for $(p_1,\dots,p_k)$ is a point $p$ that lies outside the rectangular hull of $\{p_1,\dots,p_k\}$, but which lies horizontally or vertically between $\{p_1,\dots,p_{k-1}\}$ and $p_k$. The position of the pin $p$ relative to the rectangular hull of $\{p_1,\dots,p_k\}$ naturally leads us to describe a proper pin as \emph{left}, \emph{right}, \emph{up} or \emph{down}.

A \emph{ pin permutation} is a permutation defined by a sequence of points $(p_1,\dots,p_k)$ for which each $p_i$ ($i\geq 3$) is a proper pin for $(p_1,\dots,p_{i-1})$. This definition implicitly assumes that the first two points $p_1$ and $p_2$ are specified in advance, and for our purposes it suffices to consider pin permutations initiated in the following way. We work in the $2\times 2$ gridded plane -- that is, the plane equipped with one horizontal and one vertical line, and we can take these lines to be the $x$- and $y$-axes, respectively (so that the intersection of the two lines is at the origin). Each of the four regions of our grid, therefore, corresponds to a quadrant, and these are numbered 1 to 4 in the usual manner:
\[
\begin{array}{c|c}
2&1 \\\hline	
3&4 \\
\end{array}.
\]
From now on, we will be working exclusively with permutations that are gridded in this $2\times 2$ grid. Note that~\cite{jarvis:pin-classes-i} works instead with the notion of a \emph{centered} permutation, which is a permutation equipped with an additional coordinate. There is a clear equivalence between these and $2\times 2$ gridded permutations, in which the additional coordinate of a centered permutation corresponds to the origin of a $2\times 2$ gridded permutation. In this paper, either viewpoint can be taken, and in any case we will drop the $\gridded$ superscript.

To construct a pin permutation $(p_1,p_2,\dots)$ in the $2\times 2$ grid, we use the following procedure: point $p_1$ can be placed in any of the four quadrants, and point $p_2$ is then placed so that it is a proper pin for $\{(0,0),p_1\}$. We can thus assign the point $p_2$ with one of the same four directions as subsequent pins. Subsequent points are then added according to the rules defined above, with the additional requirement that the origin $(0,0)$ belongs to each rectangular hull. In other words, we treat the origin $(0,0)$ as a `$0$th' pin. 

Pin permutations necessarily alternate their direction: if $(p_1,p_2,\dots)$ is a pin permutation and $p_i$ is an up or down (respectively, left or right) pin, then $p_{i+1}$ must be a left or right (respectively, up or down) pin in order to ensure that it separates $p_i$ from all earlier pins.  Note further that for $i>2$ the quadrant containing $p_i$ is determined by the directions of $p_{i-1}$ and $p_i$, whereas the quadrant containing $p_2$ is determined by the direction of $p_2$, and the quadrant containing $p_1$.

We note here two specific families of pin permutations: an \emph{oscillation} is a pin permutation for which each pin has one of only two directions: if these directions are right/up or down/left, then we have an \emph{increasing oscillation}, otherwise we have a \emph{decreasing oscillation}.

Given the way that a pin's direction defines its role in the construction of a pin permutation, it is natural to encode pin permutations using words over a finite alphabet. We will call the encoding that has been used most frequently in the past the \emph{basic encoding}; this encoding uses the alphabet $\{\sfl,\sfr,\sfu,\sfd\}$ to encode each pin $p_i$ for $i\geq 2$, together with a \emph{quadrant letter} from $\{1,2,3,4\}$ to encode the position of $p_1$. In this way, a pin permutation $\pi$ of length $n$ corresponds to a word $e(\pi)$ of length $n$ in the language defined by the following regular expression:
\[\{1,2,3,4\}\{\varepsilon,\sfu,\sfd\}(\{\sfl,\sfr\}\{\sfu,\sfd\})^*\cup \{1,2,3,4\}\{\varepsilon,\sfl,\sfr\}(\{\sfu,\sfd\}\{\sfl,\sfr\})^*\]
whee $\varepsilon$ denotes the empty word.

In this paper we find a slightly different encoding to be more convenient. This encoding, which we call the \emph{memory encoding}, has its origins in~\cite{bv:wqo-uncountable:} and also uses an eight-letter alphabet. However, in the memory encoding each letter also records the direction of the \emph{previous} pin. The alphabet is $\AAA_m = \{\sflu,\sfld,\sfru,\sfrd,\sful,\sfur,\sfdl,\sfdr\}$, with the convention that the letter $\sfr_\sfu$, for example, encodes a right pin that follows an up pin. Note too that the pin $p_1$ can also be encoded with precisely one of these eight letters depending upon which quadrant contains $p_1$, and in which direction (left/right or up/down) the next pin $p_2$ lies. For example, if $p_1$ is in quadrant 1 and $p_2$ is an up or down pin, then $p_1$ is encoded by the letter $\sfr_\sfu$. With the memory encoding, a pin permutation $\pi$ corresponds to a word $m(\pi)$ over $\AAA_m$ that is accepted by the following finite state automaton. In this automaton all states are both start and accept states, and all transitions inherit the label of their target state.

{\centering
\begin{tikzpicture}[scale=.55]
	\node[circle,draw,thick,inner sep=0pt] (ru) at (1,1) {\strut$\sfru$};
	\node[circle,draw,thick,inner sep=0pt] (ur) at (2.5,2.5) {\strut$\sfur$};
	\node[circle,draw,thick,inner sep=0pt] (dr) at (1,-1) {\strut$\sfdr$};
	\node[circle,draw,thick,inner sep=0pt] (rd) at (2.5,-2.5) {\strut$\sfrd$};
	\node[circle,draw,thick,inner sep=0pt] (ul) at (-1,1) {\strut$\sful$};
	\node[circle,draw,thick,inner sep=0pt] (lu) at (-2.5,2.5) {\strut$\sflu$};
	\node[circle,draw,thick,inner sep=0pt] (ld) at (-1,-1) {\strut$\sfld$};
	\node[circle,draw,thick,inner sep=0pt] (dl) at (-2.5,-2.5) {\strut$\sfdl$};

\draw[thick,->] (ru) edge[out=35,in=-125] (ur) (ur) edge[out=-145,in=55] (ru)
(dr) edge[out=-55,in=145] (rd) (rd) edge[out=125,in=-35] (dr)
(dl) edge[out=35,in=-125] (ld) (ld) edge[out=-145,in=55] (dl)
(lu) edge[out=-55,in=145] (ul) (ul) edge[out=125,in=-35] (lu);
\draw[thick,->] (ru) edge (dr) (dr) edge (ld) (ld) edge (ul) (ul) edge (ru)
	(ur) edge (lu) (lu) edge (dl) (dl) edge (rd) (rd) edge (ur);
\end{tikzpicture}\par
}
We call any word $w$ accepted by this automaton a \emph{pin word}. There is a one-to-one correspondence between pin words and pin permutations: one direction is provided by the memory encoding, while the other direction follows by placing pins according to the directions dictated in the memory encoding. Much of our later work will involve constructing the (gridded) pin permutation $\pi_w$ from a given pin word $w$.

For an example of a pin permutation and its two encodings, see Figure~\ref{fig-pin-example}.

\begin{figure}
\centering
\begin{tikzpicture}[scale=.25]
 	\plotpermgrid{12,2,5,13,8,11,9,6,10,4,7,15,3,1,14}
	\draw (7,8) -- (8)-| (11) |- (10) -| (6) |- (7) -| (4) |- (5) -| (13) |- (12) -| (2) |- (3) -| (15) |- (14) -| (1); 
	\draw[very thick] (0.5,7.5) -- ++(15,0);
	\draw[very thick] (6.5,0.5) -- ++(0,15);
\end{tikzpicture}\par 
	\caption{A pin permutation whose basic encoding is the word $1\sfl\sfu\sfr\sfd\sfr\sfd\sfl\sfu\sfl\sfd\sfr\sfu\sfr\sfd$, and whose memory encoding is $\sfur\sflu\sful\sfru\sfdr\sfrd\sfdr\sfld\sful\sflu\sfdl\sfrd\sfur\sfru\sfdr$.}\label{fig-pin-example}
\end{figure}

The memory encoding is more amenable than the basic encoding to handling subpermutations of pin permutations. For example, consider the permutation $\pi=4731526$, with basic encoding $e(\pi)=1\sfu\sfr\sfu\sfl\sfu\sfr$ and memory encoding $m(\pi)=\sfru\sfur\sfru\sfur\sflu\sful\sfru$. The permutation $25314$ can be obtained by deleting the first two points of the sequence, and this can be witnessed in the memory encoding simply by deleting the first two letters of $m(\pi)$. On the other hand, to obtain the basic encoding of $25314$ from $e(\pi)$, one must delete the first two letters and replace the third with an appropriate numeral. This complication can be handled by a transducer (as used in~\citet*{brignall:simple-permutat:b}), but here we have chosen to avoid this.%it proves more convenient to work with the memory encoding. 

%Thus, in the sequel we will principally use the memory encoding $m(\pi)$ of $\pi$. 
Note that in the memory encoding, if the predecessor letter of some given letter is known, then the subscript on the given letter is redundant. For ease of presentation, we will frequently exploit this fact and suppress the subscripts used in the memory encoding except where this redundancy does not occur. 

Now let $w=w_1w_2\cdots$ be an \emph{infinite} pin word (or a \emph{pin sequence}). The infinite (gridded) permutation $\pi_w$ can be constructed in the same way as for finite permutations. The \emph{pin class} corresponding to $\pi_w$ consists of all the finite subpermutations of $\pi_w$:
\[
\C_w = \{ \sigma : \sigma \leq \pi_w\}.
\]
Note here that we are considering containment $\leq$ of \emph{gridded} permutations, as defined earlier. Note further that, although we will be working exclusively with gridded classes $\C_w$, the `ungridded' versions of such classes form permutation classes, and of course both the gridded and un-gridded versions have the same growth rate by Lemma~\ref{lem-gr-gridded}.

\paragraph{Box sums} A widely-used tool in the study of permutation classes is the notion of direct (and skew) sums. Here, we will use the same generalisation of this notion to the $2\times 2$ grid as introduced by~\cite{jarvis:pin-classes-i}. 
%This generalisation is essentially the same as `inflating the origin' in~\cite{bv:wqo-uncountable:}, and a special case of the `$M$-sums' in~\cite{bbr:unicyclicgrids:}. 
Let $\sigma$ and $\tau$ be two permutations in the $2\times 2$ gridded plane. The \emph{box sum} of $\sigma$ and $\tau$, denoted $\sigma \boxplus  \tau$, is the gridded permutation of length $|\sigma|+|\tau|$ formed by inserting a copy of $\sigma$ at the origin of a copy of $\tau$. If we denote the points of $\sigma$ and $\tau$ in quadrant $i=1,2,3,4$ by $\sigma_i$ and $\tau_i$, then the box sum can be viewed graphically as follows.
\[
\begin{tikzpicture}[scale=0.2, baseline=-3pt]
	\draw (-3,0) -- ++ (6,0) (0,-3)--++(0,6);
	\foreach \x/\y [count=\n] in {1.5/1.5,-1.5/1.5, -1.5/-1.5, 1.5/-1.5}
		\node at (\x,\y) {$\sigma_\n$};
\end{tikzpicture}
\ \boxplus\ 
\begin{tikzpicture}[scale=0.2, baseline=-3pt]
	\draw (-3,0) -- ++ (6,0) (0,-3)--++(0,6);
	\foreach \x/\y [count=\n] in {1.5/1.5,-1.5/1.5, -1.5/-1.5, 1.5/-1.5}
		\node at (\x,\y) {$\tau_\n$};
\end{tikzpicture}
\ =\ 
\begin{tikzpicture}[scale=0.2, baseline=-3pt]
	\draw[dashed,gray] (-6,-3) -- ++ (12,0) (-6,3) -- ++ (12,0) (-3,-6)--++(0,12) (3,-6)--++(0,12);
	\draw (-6,0) -- ++ (12,0) (0,-6)--++(0,12);
	\foreach \x/\y [count=\n] in {1.5/1.5,-1.5/1.5, -1.5/-1.5, 1.5/-1.5} {
		\node at (\x,\y) {$\sigma_\n$};
		\node at (3*\x,3*\y) {$\tau_\n$};
	}
\end{tikzpicture}
\]
A gridded permutation is \emph{box decomposable} or $\boxplus$-decomposable if it can be expressed as the box sum of two smaller permutations, and \emph{box indecomposable} (or $\boxplus$-indecomposable) otherwise. Graphically, box decomposability can be recognised by a set of points whose rectangular hull includes the origin and has the property that there are no points in the regions above, below, to the left or to the right of this hull. 

As with direct and skew sums, one can decompose a permutation $\pi$ into a $\boxplus$-sum of box indecomposable permutations, and the collection $S$ of indecomposable permutations in this sum is unique. However, unlike direct and skew sums, it is possible for there to be more than one expression for $\pi$ using the collection $S$. For example, the gridded permutation \begin{tikzpicture}[scale=0.08, baseline=-3pt]
	\draw (-3,0) -- ++ (6,0) (0,-3)--++(0,6);
	\smallgrid
	\foreach \x/\y in {-1/-1,2/1,1/2}
		\node[permpt] at (\x,\y) {};
\end{tikzpicture}\ has two expressions as a box sum of \netwo\ and \swpt\ as follows:
\[
\begin{tikzpicture}[scale=0.2, baseline=-3pt]
	\draw (-3,0) -- ++ (6,0) (0,-3)--++(0,6);
	\foreach \x/\y in {-1/-1,2/1,1/2}
		\node[permpt] at (\x,\y) {};
\end{tikzpicture}
\quad =\quad
\begin{tikzpicture}[scale=0.2, baseline=-3pt]
	\draw (-3,0) -- ++ (6,0) (0,-3)--++(0,6);
	\foreach \x/\y in {-1/-1}
		\node[permpt] at (\x,\y) {};
\end{tikzpicture}
\boxplus
\begin{tikzpicture}[scale=0.2, baseline=-3pt]
	\draw (-3,0) -- ++ (6,0) (0,-3)--++(0,6);
	\foreach \x/\y in {2/1,1/2}
		\node[permpt] at (\x,\y) {};
\end{tikzpicture}
\quad =\quad
\begin{tikzpicture}[scale=0.2, baseline=-3pt]
	\draw (-3,0) -- ++ (6,0) (0,-3)--++(0,6);
	\foreach \x/\y in {2/1,1/2}
		\node[permpt] at (\x,\y) {};
\end{tikzpicture}
\boxplus
\begin{tikzpicture}[scale=0.2, baseline=-3pt]
	\draw (-3,0) -- ++ (6,0) (0,-3)--++(0,6);
	\foreach \x/\y in {-1/-1}
		\node[permpt] at (\x,\y) {};
\end{tikzpicture}.
\]
It is clear that a sum of two $\boxplus$-indecomposable permutations, each of which is wholly contained in opposing quadrants, can always be taken in either order. We say that such a pair of $\boxplus$-indecomposable permutations \emph{commute}. In fact, we have:

\begin{obs}[See~\cite{jarvis:pin-classes-i}]
Let $\sigma$ and $\tau$ be a pair of $\boxplus$-indecomposable gridded permutations. Then $\sigma$ and $\tau$ commute if and only if $\sigma=\tau$, or $\sigma$ and $\tau$ are wholly contained in opposing quadrants.
\end{obs}

To see this, note that if $\sigma$ and $\tau$ are not wholly contained in opposing quadrants, then in $\sigma\boxplus\tau$ there is a point $p$ of $\sigma$ that lies horizontally or vertically between the origin and some point $q$ of $\tau$. Similarly, in $\tau\boxplus\sigma$ the points corresponding to $p$ and $q$ must be in the other order: that is, $q$ separates $p$ from the origin. Since both $\sigma$ and $\tau$ are $\boxplus$-indecomposable, we can only have $\sigma\boxplus\tau = \tau\boxplus\sigma$ if $\sigma$ can be embedded in $\tau$, and vice versa. That is, $\sigma=\tau$.

\paragraph{Box closed classes and enumeration}
A (gridded) class $\C$ is \emph{box closed} (or $\boxplus$-closed) if $\sigma\boxplus\tau\in\C$ for every $\sigma,\tau\in\C$, and the \emph{box closure} of a set $\DDD$ is the smallest box-closed class containing $\DDD$, and is denoted $\bigboxplus \DDD$.  
These concepts are natural extensions of similar ideas for sum and skew closed (ungridded) classes, but whereas for (say) a sum-closed class we can write $f(z)=1/(1-g(z))$ for the relationship between the generating function of the class $f$ and that of the sum-indecomposable permutations $g$, the corresponding expression for box-closed classes would result in overcounting permutations whose decomposition into box indecomposable permutations involves adjacent commuting pairs.

To handle commuting pairs, we require a result originally due to~\cite{cartier:problemes-combi:}, and first applied to permutation classes by~\cite[Lemma 7]{bevan:growth-rates-ggc}. Given a finite alphabet $\AAA$ and a collection of \emph{commutation rules} (that is, unordered pairs of letters that may be read in either order), the \emph{trace monoid} $\MMM(\AAA)$ is the set of equivalence classes of words over $\AAA$, under the equivalence relation determined by the commutation rules.

\begin{lemma}[{\cite{cartier:problemes-combi:}; see also~\cite[Note V.10]{flajolet:analytic-combin:}}]\label{lem-cartier-foata}
Let $\AAA$ be a finite alphabet and let $C$ be a set of commutation rules on elements of $\AAA$.
Then the trace monoid $\MMM(\AAA)$ has generating function
\[
M(z) = \left(\sum_F(-1)^{|F|}z^{|F|}\right)^{-1}
\]
where the sum is over all sets $F$ composed of distinct letters that commute pairwise.
\end{lemma}

It is possible to associate each letter of $\AAA$ with a polynomial or power series, from which we obtain the following corollary. 

\begin{cor}\label{cor-cartier-foata}
	Let $\AAA$ be a finite alphabet and let $g:\AAA \to\mathbb{Z}[\![z]\!]$ be any mapping. Then the trace monoid $\MMM(\AAA)$ has weighted generating function
	\[
	M_g(z) = \left(\sum_{F} (-1)^{|F|}\prod_{a\in F}g(a)\right)^{-1}
	\]
where the sum is over all sets $F$ of distinct letters that commute pairwise.
\end{cor}

In the case of the four-quadrant grids in this paper, we use a five letter alphabet $\{a,\ell_1,\ell_2,\ell_3,\ell_4\}$, and two commutation relations, $\ell_1\ell_3=\ell_3\ell_1$ and $\ell_2\ell_4=\ell_4\ell_2$.
For a $\boxplus$-closed class $\C$, define the function $g:\{a,\ell_1,\ell_2,\ell_3,\ell_4\}\to \mathbb{Q}(z)$ so that $g(\ell_i)$ is the generating function for the $\boxplus$-indecomposable permutations wholly contained in quadrant $i$, and $g(a)$ is the generating function for all the other $\boxplus$-indecomposable permutations. By Corollary~\ref{cor-cartier-foata}, the generating function for $\C$ is
\[
f_\C(z) = \frac1{1 - \left(g(a) + g(\ell_1) + g(\ell_2) + g(\ell_3) + g(\ell_4)\right) + g(\ell_1)g(\ell_3) + g(\ell_2)g(\ell_4)}.
\]
By way of an example, consider $\XXX^\gridded = \bigboxplus\left\{\nept,\,\nwpt,\,\swpt,\,\sept\right\}$. We take $g(a)=0$ and $g(\ell_i)=z$ for $i=1,2,3,4$, which gives us the generating function
\[
f_{\XXX^\gridded}(z) = \frac1{1-4z+2z^2}.
\]
This class is precisely the class of gridded permutations that can be drawn on a geometric `$X$'. The generating function for the ungridded version of $\XXX$ is (see~\cite{waton:on-permutation-:} or~\cite{elizalde:the-x-class-and:})
\[\frac{1-3z}{1-4z+2z^2}.\]
Note that the denominator agrees with that for $\XXX^\gridded$, confirming that $\XXX$ and $\XXX^\gridded$ have the same growth rate as required by Lemma~\ref{lem-gr-gridded}.

We now briefly turn our attention to general theory regarding the growth rates of classes. First, for a permutation class $\C$ that is sum- or skew-closed,~\cite{arratia:on-the-stanley-:} observed that the sequence $(|\C_n|)$ is \emph{supermultiplicative} (that is, $|\C_m||\C_n|\leq|\C_{m+n}|$), and therefore Fekete's lemma implies that $\gr(\C)$ exists. This fact can alternatively be derived from the \emph{supercritical sequence schema} described in Section V.2 of~\cite{flajolet:analytic-combin:}. 

Neither approach works directly for $\boxplus$-closed classes, because of the expressions provided by Lemma~\ref{lem-cartier-foata}: for example, the counting sequence for the $\XXX$ class is not supermultiplicative, nor does it satisfy the positivity condition for the supercritical sequence schema. Nevertheless, it is of course still the case that the geometric `X' has growth rate $1/\rho$, where $z=\rho$ is the smallest real positive solution to $4z-2z^2=1$, since $\rho$ is the unique dominant singularity of $f_{\XXX^\gridded}(z)$. 

Restricting our attention to pin classes, the previous paper established the following result.

\begin{thm}[\cite{jarvis:pin-classes-i}]
For any pin class $\C$, the growth rate $\gr(\C)$ exists.
\end{thm}

\paragraph{Combinatorics of Words}
We have already discussed several examples of words over finite alphabets, but here we briefly survey some definitions and results from the standard literature on the combinatorics of words that we will require. Further details about the concepts here can be found in the survey article by~\cite{cassaigne:factor-complexity:}, and references therein.

Given a (finite or infinite) word $w$, a \emph{factor} of $w$ is a contiguous subsequence of $w$. An infinite word $w$ over a finite alphabet is said to be:
\begin{itemize}
\item \emph{periodic} if there exists $k$ such that $w_i = w_{k+i}$ for all $i\geq 1$. The smallest such value of $k$ for which this holds is called the \emph{period} of $w$,
\item \emph{eventually periodic} if $w=uw'$ for some finite word $u$ and periodic word $w'$. The \emph{period} of $w$ is defined to be equal to the period of $w'$.
\item \emph{recurrent} if every finite factor of $w$ appears infinitely often in $w$, and
\item \emph{eventually recurrent} if $w=uw'$ for some finite word $u$ and recurrent word $w'$.	
\end{itemize}
Clearly, every periodic (respectively, eventually periodic) word $w$ is recurrent (resp., eventually recurrent). 

Given an infinite word $w$, the \emph{complexity} of $w$ is the function $p_w(n)$ that counts the number of distinct factors of length $n$ that appear in $w$. Since every factor of $w$ of length $n$ can be extended in at least one way by adding a letter on the right hand end, there is a natural injection from factors of length $n$ to factors of length $n+1$, for all $n$. Consequently, we have.

\begin{prop}[See~\cite{cassaigne:factor-complexity:}]\label{prop-p-increasing}
	For an infinite word $w$, the complexity $p_w$ is increasing.
\end{prop}

In the case of words that are not periodic, we have the following result which dates back to the 1930s.

\begin{thm}[\cite{morse:symbolic-dynami:}]\label{thm-morse-hedlund}
	For an infinite word $w$ that is not periodic or eventually periodic, $p_w$ is strictly increasing.
\end{thm}

By contrast, words that are periodic or eventually periodic must have a complexity function that is constant for suitably large $n$. In fact, closer analysis yields the following.

\begin{lemma}[Essentially due to~\cite{morse:symbolic-dynami:}]\label{lem-constant-forever}
For an infinite word $w$ over a finite alphabet, if there exists $N$ such that $p_w(N) = p_w(N+1)$, then $p_w(n) = p_w(N)$ for all $n\geq N$.
\end{lemma}

\begin{proof}
For each $n\geq 1$, let $\rho_n$ denote any natural injection from the factors of length $n$ in $w$ to the factors of length $n+1$ in $w$, in which a single letter is added to the right hand end of each factor. We will be done if we can show that $\rho_n$ is a bijection between the factors of lengths $n$ and $n+1$, for all $n\geq N$, and we show this by induction.  

The base case is clear, so fix $k\geq N$ for $\rho_k$. 
Take a factor $a=a_1\cdots a_{k+1}$ of $w$ of length $k+1$, and let $a_{k+2}$ be the letter such that $\rho_{k+1}(a) = a_1\cdots a_{k+2}$. Suppose that $b$ is any letter for which $a_1\cdots a_{k+1}b$ is a factor of $w$. 

Let $a' = a_2\cdots a_{k+1}$. The words $a'a_{k+2}$ and $a'b$ are both factors of $w$ of length $k+1$ formed by extending $a'$ to its right. Since $\rho_k$ is a bijection, there is only one letter that can be used to extend $a'$ to the right. Thus $\rho_k(a')=a'a_{k+2}=a'b$ so $a_{k+2}=b$. This shows that $a$ extends uniquely to the right, and so $\rho_{k+1}$ is a bijection between the factors of lengths $k+1$ and $k+2$, as required.
\end{proof}

Let $w$ be an infinite word over a finite alphabet $A$. If a finite factor $u$ of $w$ appears infinitely often, we say that $u$ is a \emph{recurrent factor}, otherwise it is a \emph{non-recurrent factor}. The \emph{recurrent complexity} is the function $q_w(n)$ that counts the number of distinct recurrent factors of length $n$ in $w$. Although this notion seems natural enough, it does not seem to have appeared in the literature previously. Clearly $p_w(n)\geq q_w(n)$ for all $n$, and $p_w=q_w$ if and only if $w$ is recurrent. We observe the following.

\begin{thm}\label{thm-recurrent-morse-hedlund}
	For an infinite word $w$, $q_w$ is increasing. Furthermore, if $w$ is not eventually periodic, then $q_w$ is strictly increasing. 
\end{thm}

\begin{proof}
 For the first part, let $n$ be arbitrary, and let $a$ be any recurrent factor of length $n$. Since $w$ is a word over a finite alphabet $A$, say, there exists at least one symbol $x\in A$ such that the letter after infinitely many instances of the word $a$ is $x$. Thus, $ax$ is a recurrent factor of length $n+1$, which establishes an injection from the set of recurrent factors of length $n$ to those of length $n+1$.
 
 For the second part, suppose that $q_w(n) = q_w(n+1)$ for some $n$. Write $w = uw'$, where $u$ is a finite prefix that has been chosen so that $w'$ contains none of the non-recurrent factors of length up to $n+1$ that appeared in $w$. Note that $q_w(k)=q_{w'}(k)$ for all $k$. Furthermore, by construction we have $q_{w'}(k)=p_{w'}(k)$ for all $k\leq n+1$. Thus
\[
p_{w'}(n) = q_{w'}(n) = q_w(n) = q_w(n+1) = q_{w'}(n+1) = p_{w'}(n+1),
\]
which shows that $w'$ is eventually periodic by Theorem~\ref{thm-morse-hedlund}. However, since $w=uw'$, it follows that $w$ is also eventually periodic.
\end{proof}

\begin{lemma}\label{lem-q-constant-forever}
For an infinite word $w$ over a finite alphabet, if there exists $N$ such that $q_w(N) = q_w(N+1)$, then $q_w(n) = q_w(N)$ for all $n\geq N$.
\end{lemma}

We omit the proof of this lemma, since it is essentially the same as that given in Lemma~\ref{lem-constant-forever}.

%
%
%
%
%
%
%
%
%%%%%%%%%%%%
\section{Pin class structure}\label{sec-pin-class-structure}

We are now in a position to begin to describe the structure of permutations in pin classes. By definition, a pin permutation $(p_1,\dots,p_n)$ has the property that the set of points $\{p_1,\dots,p_i\}$ for any $i\geq 2$ cannot form a component for a box decomposition. On the other hand, the pin $p_{i}$ is the only point of the pin permutation that cuts the rectangular hull of $\{p_1,\dots,p_{i-1}\}$, thus its removal renders a box decomposition.

\begin{obs}\label{obs-deleting-pin}
For $n\geq 3$, let $u=u_1\cdots u_n$ be a pin word and $\pi_u=(p_1,\dots,p_n)$ the associated pin permutation. For any $1<i<n$, the subpermutation $\pi'$ formed by deleting the point $p_i$ from $\pi_u$ satisfies
\[
\pi' = \pi_{u_1\cdots u_{i-1}} \boxplus \pi_{u_{i+1}\cdots u_n}.
\]
\end{obs}

It is natural to wonder whether every pin permutation gives rise to a $\boxplus$-indecomposable permutation. Unfortunately, this is not quite true, as the following example illustrates.
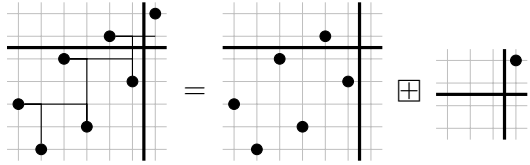
\[
\begin{tikzpicture}[scale=.3,baseline=28pt]
	\plotpermgrid{3,1,5,2,6,4,7}
	\draw (7,6) -- (6)-| (4) |- (5) -| (2) |- (3) -| (1); 
	\draw[very thick] (0.5,5.5) -- ++(7,0);
	\draw[very thick] (6.5,0.5) -- ++(0,7);
\end{tikzpicture}
\ =\  
\begin{tikzpicture}[scale=.3,baseline=28pt]
	\plotpermgrid{3,1,5,2,6,4,0}
	\draw[very thick] (0.5,5.5) -- ++(7,0);
	\draw[very thick] (6.5,0.5) -- ++(0,7);
\end{tikzpicture}
\ \boxplus\ 
\begin{tikzpicture}[scale=.3,baseline=20pt]
	\plotpermgrid{0,0,0,4}
	\draw[very thick] (0.5,2.5) -- ++(4,0);
	\draw[very thick] (3.5,0.5) -- ++(0,4);
\end{tikzpicture}
\]
\cite{jarvis:pin-classes-i} establishes a complete classification of the box-decomposable pin permutations. We will return to this issue in Section~\ref{sec-two-quadrants}, but for now it suffices to note the following.

\begin{prop}\label{prop-indecomposable-is-factor}
Let $w$ be an infinite pin word, and let $\pi\in\C_w$ be a box-indecomposable permutation. Then $\pi = \pi_{w'}$ for some factor $w'$ of $w$.
\end{prop}

\begin{proof}
Take any box indecomposable permutation $\pi\in\C_w$. We can witness the membership of $\pi$ in $\C_w$ by embedding its points in that of some pin permutation $(p_1,\dots,p_{N})$ for some $N\geq |\pi|$, which we may take to be as small as possible. In turn, this pin permutation corresponds to a factor $w_M\cdots w_{M+N-1}$ of $w$, for some $M$.

If $N > |\pi|$, then since $N$ was taken to be as small as possible, this implies that any given embedding of $\pi$ in $(p_1,\dots,p_N)$ must use both $p_1$ and $p_N$, and hence there is some $p_i$, with $1<i<N$, that is \emph{not} used in the embedding. Observation~\ref{obs-deleting-pin} now implies that $\pi$ embeds in $\pi_{w_M\cdots w_{M+i-2}}\boxplus \pi_{w_{M+i}\cdots w_{M+N-1}}$, with at least one point from each of the two box sum components used in the embedding. This implies that $\pi$ is box decomposable, but this is a contradiction. Hence $N=|\pi|$ and the proof is complete.
\end{proof}

In the context of pin words, we can now derive the following useful result.

\begin{prop}\label{prop-w-recurrent}
If $w$ is a recurrent pin word, then $C_w$ is $\boxplus$-closed.	
\end{prop}

\begin{proof}
It suffices to show that if $\sigma,\tau\in\C_w$ with the property that $\tau$ is $\boxplus$-indecomposable, then $\sigma\boxplus\tau\in\C_w$.

Let $u$ be any finite prefix of $w$ such that $\sigma$ embeds in $\pi_u$. Then we can write $w=uw'$.
By Proposition~\ref{prop-indecomposable-is-factor}, the $\boxplus$-indecomposable permutation $\tau$ in $\C$ can be constructed using a factor of $w$. Since $w$ is recurrent, this factor appears infinitely often. In particular, we can find this factor in the word $w'$ in such a way as it does not use the first letter of $w'$, and the embedding of $\sigma\boxplus\tau$ in $\pi_w$ now follows.
\end{proof}

In cases where the defining word $w$ is not recurrent, we can recover a subclass of $\C_w$ that \emph{is} $\boxplus$-closed. The \emph{box-interior} of the class $\C_w$ is defined by
\[
\C_w^\boxplus = \{ \sigma \in \C_w : \sigma \leq \pi_u \text{ for some recurrent factor }u\}.
\]
Clearly $\C_w^\boxplus \subseteq \C_w$ is a permutation class, and $\C_w^\boxplus = \C_w$ if and only if $w$ is recurrent (this follows by an argument similar to that given in the proof of Proposition~\ref{prop-w-recurrent}). However, even more is true:

\begin{thm}[\cite{jarvis:pin-classes-i}]\label{thm-jarvis-box-interior}
For any pin word $w$, we have $\gr(\C_w) = \gr(\C_w^\boxplus)$.	
\end{thm}

%
%
%
%
%
%
%%%%%%%%%%%%%%%%
\subsection{Visits}

Given a pin word $w$, a \emph{visit} to a quadrant $q$ is a finite contiguous subsequence $w_{i+1}\cdots w_{i+k}$ of $w$ such that the $k$ corresponding points of $\pi_w$ are placed in quadrant $q$, while neither the point corresponding to $w_i$ (if it exists) nor $w_{i+k+1}$ belongs to quadrant $q$. Trivially, we have:

\begin{prop}\label{prop-visits-are-oscillations}
The permutation corresponding to any visit forms an oscillation, which is increasing if it belongs to quadrants 1 or 3, and decreasing otherwise.
\end{prop}

Each visit $w_{i+1}\cdots w_{i+k}$ is initiated by an \emph{arrival}, being the first point of the visit (corresponding to $w_{i+1}$), and a \emph{departure}, being the last point (corresponding to $w_{i+k}$). Our work on establishing growth rates will, in part, rely on the analysis of the number and length of visits to different quadrants.

Note the requirement that a visit is a \emph{finite} subsequence. Thus, it is possible for an infinite pin sequence to contain only finitely many visits, or possibly none at all. The number of quadrants which a pin sequence visits infinitely often strongly controls the range of growth rates of the corresponding pin class. Before we demonstrate this, we first show that pin sequences can be freely truncated at the start without impacting the growth rate.

\begin{prop}\label{prop-truncate-gr}
For any pin sequence $w$, and any expression $w=vw'$ where $v$ is a finite prefix, we have $\gr(\C_w)=\gr(\C_{w'})$.	
\end{prop}

\begin{proof}
Since $v$ is finite, the pin classes $\C_w$ and $\C_{w'}$ have the property that their box interiors are identical. The statement follows by Theorem~\ref{thm-jarvis-box-interior}.
\end{proof}

A particularly useful application of this proposition is the following.

\begin{prop}\label{prop-kill-finite-visits}
For any pin sequence $w$ that visits quadrant $q$ only finitely many times, we have $\gr(\C_w)=\gr(\C_{w'})$ where $w=vw'$ and $w'$ never visits quadrant $q$.	
\end{prop}

We now begin our classification of small growth rates.

\begin{thm}\label{thm-finite-visits-gr}
Let $w$ be an infinite pin word that contains only finitely many visits. Then $\gr (\C_w) = \kappa \approx 2.20557$, the largest real root of $z^3-2z^2-1$.
\end{thm}

\begin{proof}
	By Proposition~\ref{prop-kill-finite-visits}, we can, by truncating the word if necessary, assume that $w$ in fact contains \emph{no} visits. This means that $w$ belongs to a single quadrant and is thus periodic.  
	
	Without loss generality we have $w = (\sfur\sfru)^\infty$, and the pin class $\C_{w}$ is $\boxplus$-closed by Proposition~\ref{prop-w-recurrent}. By inspection, the number of $\boxplus$-indecomposable permutations in $\C_{w}$ of lengths $1,2,3,\dots$ are $1,1,2,2,\dots$, and thus $\C_{w}$ has generating function
	\[
	f_{C_{w}}(z) = \frac1{1-\frac{z+z^3}{1-z}} = \frac{1-z}{1-2z-z^3}.
	\]
	Thus $\gr(\C_{w}) = \kappa \approx 2.20557$, the largest real root of $z^3-2z^2-1$.
\end{proof}

Any pin word that contains arbitrarily many visits necessarily must visit at least \emph{two} quadrants infinitely often. Our next task is to establish that any pin class whose defining sequence visits more than two quadrants infinitely often has growth rate well above our target growth rate of $\mu\approx 3.28277$. 

\begin{prop}\label{prop-4-quadrants}
Let $w$ be an infinite pin word, and suppose that $\C_w^\gridded$ visits every quadrant infinitely often. Then $\gr(\C^\gridded_w)\geq 2+\sqrt2 \approx 3.414$.
\end{prop}

\begin{proof}
Any class that visits every quadrant infinitely often must contain the box closure 
\[
	\bigboxplus\left\{\nept,\,\nwpt,\,\swpt,\,\sept\right\},
\] 
which is the class $\XXX$ considered at the end of Section~\ref{sec-preliminaries}. The generating function of $\XXX^\gridded$ has denominator $1-4z+2z^2$, and thus $\gr(\XXX) = 2+\sqrt2$, from which the proposition follows.
\end{proof}

Note that, in fact, the \emph{smallest} growth rate of a pin class in all four quadrants is approximately $3.48806$ (being the largest real root of $z^5-5z^4+6z^3 -2z^2 -z-3$) and is achieved by the `Widdershins spiral' defined by the word $w=(\sfur\sfl\sfd\sfr)^\infty$. To see this, consider the following class:
\[
	\bigboxplus\left\{\nept,\,\nwpt,\,\swpt,\,\sept,\,\netwo\right\}
\]
The generating function for this class has denominator $1-4z+z^2+z^3$, and thus has growth rate approximately equal to $3.69109$. Consequently, by considering all symmetries of the above class, the smallest pin class in all four quadrants must be defined by a pin sequence which has the property that, except for finitely many exceptions, every visit is formed of a single letter. It is now straightforward to argue that the only pin sequences that are permissible are $w=(\sfur\sfl\sfd\sfr)^\infty$ and its `reverse', $w=(\sful\sfr\sfd\sfl)^\infty$.

We now establish a lower bound for the growth rates of pin classes that visit three quadrants. Let $\lambda \approx 3.28481$ be the largest positive zero of $z^5-3z^4-z^3+z-1$.

\begin{prop}\label{prop-3-quadrants}
Let $w$ be an infinite pin sequence, such that $\C_w^\gridded$ visits precisely three quadrants infinitely often. Then $\gr(\C^\gridded_w) \geq \lambda$.
\end{prop}

\begin{proof}
Without loss of generality, suppose that $\C^\gridded_w$ visits quadrants 1--3 infinitely often. By truncating and appealing to Proposition~\ref{prop-kill-finite-visits} as necessary, without loss of generality we can assume that $w$ never visits quadrant 4.

We first claim that the following $\boxplus$-indecomposable permutations belong to the box-interior, $\C_w^\boxplus$:
\[
\nept,\ \nwpt,\ \swpt,\ \netwo,\ \swtwo,\ \nthree,\ \wthree
\]
This claim is clear for the three singletons (since $w$ visits all of quadrants 1--3 infinitely often), so now consider any index $k$ such that $w_k$ is an arrival in quadrant 1. This arrival must be from quadrant 2, so $w_{k-1} = \sful$ and $w_k=\sfru$. Furthermore, the next letter $w_{k+1}$ must also belong to quadrant 1, hence $w_{k+1}=\sfur$. Thus, every arrival into quadrant 1 defines a copy of \nthree\ (and hence also \netwo), and there are arbitrarily many such arrivals so these patterns belong to $\C_w^\boxplus$. A similar argument applies in quadrant 3 for the other two patterns.

Next, suppose that \nwtwo\ $\in\C_w^\boxplus$. Then $\C_w^\boxplus$ contains the $\boxplus$-closure of the following eight permutations:
\[
\nept,\ \nwpt,\ \swpt,\ \netwo,\ \nwtwo,\ \swtwo,\ \nthree,\ \wthree
\]
To find the generating function of the $\boxplus$-closure of these permutations, we note that the box indecomposables in quadrant 1 have generating function $z+z^2$, as do those in quadrant 3. By Corollary~\ref{cor-cartier-foata}, the generating function of this $\boxplus$-closure is therefore
\[
\frac1{1-(3z+3z^2+2z^3)+(z+z^2)^2} = \frac1{1-3z-2z^2+z^4},
\]
from which we can see that $\gr(\C_w) \geq 3.542$ (the largest real root of $z^4-3z^3-2z^2+1$).

Thus, we may now assume that $\C_w^\boxplus$ does not contain \nwtwo. By a similar argument to that used earlier, we can find arbitrarily many indices $k$ such that $w_{k-2}w_{k-1}w_kw_{k+1} w_{k+2} = \sfdl\sfld\sful\sfru\sfur$, and hence \nwfive\ $\in\C_w^\boxplus$. 

Thus, in this case $\C_w^\boxplus$ contains the following eight permutations.
\[
\nept,\ \nwpt,\ \swpt,\ \netwo,\ \swtwo,\ \nthree,\ \wthree,\ \nwfive.
\]
The generating function of this $\boxplus$-closure is
\[
\frac1{1-(3z+2z^2+2z^3+z^5)+(z+z^2)^2} = \frac1{1-3z-z^2+z^4-z^5},
\]
and therefore $\gr(\C_w) \geq \lambda$ (the largest real root of $z^5-3z^4-z^3+z-1$). 
\end{proof}

In fact, the \emph{smallest} pin class in three quadrants is given by the word $w=(\sfur\sfl\sfd\sfl\sfu\sfr)^\infty$ and has growth rate approximately 3.36132 (equal to the largest real root of $z^5-4z^4+2z^3+z^2-2z+1$). The details of the argument that establishes this are similar to, but more intricate than, that given above, and requires some additional theory presented in~\cite{jarvis:pin-classes-i}.

%
%
%
%
%
%
%%%%%%%%%%%%%%%
\section{Visits to two quadrants and periodic words}\label{sec-two-quadrants}

We have seen that when a pin sequence contains only finitely many visits, or contains arbitrarily many visits to three or four quadrants, then the growth rate of the corresponding pin class must equal $\kappa$ or exceed $\lambda \approx 3.28481$, respectively. 

It remains to consider the possible growth rates for pin classes whose sequences visit precisely two quadrants infinitely often.  By symmetry and Proposition~\ref{prop-kill-finite-visits}, we can restrict our attention to pin classes whose sequences only ever visit quadrants 1 and 2, and do so infinitely often.

Before we go further, we need to make the relationship between factors of a pin sequence and the $\boxplus$-indecomposable permutations more precise. First, as mentioned earlier there exist pin words that define $\boxplus$-decomposable permutations. This presents no problem for sufficiently long pin words in two quadrants:

\begin{lemma}[Lemma 6.1 of~\cite{bv:wqo-uncountable:}, see also~\cite{jarvis:pin-classes-i}]\label{lem-two-quadrant-indecs}
Let $w$ be a pin sequence that visits precisely two quadrants, and let $f$ be a finite factor of $w$ of length at least 3. Then $\pi_f$ is $\boxplus$-indecomposable. 
\end{lemma}

From an enumerative perspective, we would also like there to be a one-to-one correspondence between pin words appearing as factors in a pin sequence $w$, and box-indecomposable permutations in the pin class $\C_w$. In general this is not true -- for a full classification, see~\cite{jarvis:pin-classes-i} -- but in two quadrants it again holds for sufficiently long pin words.

\begin{lemma}[Proposition 6.3 of~\cite{bv:wqo-uncountable:}]\label{lem-two-quadrant-no-collisions}
Let $w$ be a pin sequence that visits precisely two quadrants, and let $f$ and $f'$ be finite factors of $w$ of length at least 4. If $\pi_f = \pi_{f'}$, then $f=f'$. 
\end{lemma}

%Thus, from length 4 onwards, there is a one-to-one correspondence between factors of the pin sequence $w$, and $\boxplus$-indecomposable permutations in $\C_w$. Similarly, there is a one-to-one correspondence between the recurrent factors of $w$, and the $\boxplus$-indecomposable permutations in $\C_w^\boxplus$.

%
%
%%%%%
\subsection{Pin sequences from binary words}\label{subsec-bin-to-pin-words}

As we are now restricting our attention to pin sequences that visit precisely two quadrants, we can introduce a method to construct these pin sequences from binary words. This construction is essentially the same as that used by~\cite{bv:wqo-uncountable:}, but here we undertake a slightly more careful enumerative analysis.

A pin sequence that is wholly contained in quadrants 1 and 2 must comprise a series of left and right steps, interleaved with up steps. Given an infinite binary sequence $b$, there is therefore a natural injection $\phi$ to a two-quadrant pin sequence $w=\phi(b)$, via $\phi(0)=\sfl\sful$ and $\phi(1)=\sfr\sfur$. Note that $\phi$ can never produce a pin sequence that begins with $\sful$ or $\sfur$, and thus in general it is not surjective. 

However, since we are primarily concerned with growth rates of pin classes, note that for any pin sequence $w$ that begins with $\sful$ or $\sfur$, we can simply remove the first letter to form the pin sequence $w'$, and then $\gr(\C_w)=\gr(\C_{w'})$ by Proposition~\ref{prop-truncate-gr}. Consequently, we can without loss of generality restrict our attention to pin sequences that are the images under $\phi$ of binary sequences.

We now evaluate the complexity and recurrent complexity of $w=\phi(b)$, given the complexity $p_b$ and recurrent complexity $q_b$ of $b$, and thus count the number of $\boxplus$-indecomposable permutations of each length in $\C_w$ and $\C_w^\boxplus$. This argument is essentially the same as that given by~\cite{jarvis:pin-classes-i}, but we include it here for completeness.

For a binary factor $a$ of length $n$ appearing in $b$, $\phi(a)$ is a pin word of length $2n$ appearing as a factor of $w=\phi(b)$. In addition, if a pin word $x$ is formed by removing the first and/or last letters of $\phi(a)$, then (trivially) $x$ is also a factor of $\phi(b)$. Furthermore, the only occurrences of $x$ as a factor in $\phi(b)$ are those that embed in an occurrence of $\phi(a)$. Thus, each factor $a$ of length $n\geq 2$ in the binary word $b$ accounts for one word $\phi(a)$ of length $2n$, two words of length $2n-1$ (delete the first or last letter of $\phi(a)$), and one word of length $2n-2$ (delete both the first and last letters of $\phi(a)$). These considerations yield:

\begin{prop}\label{prop-bin-to-pin}
Let $b$ be an infinite binary sequence, and let $w = \phi(b)$. Then, for $n\geq 1$,
\begin{gather*}
p_w(n) = 
 \begin{cases}
  p_b(k) + p_b(k+1) & n=2k\\
  2p_b(k+1) &n=2k+1	
 \end{cases}\\
q_w(n) = 
 \begin{cases}
  q_b(k) + q_b(k+1) & n=2k\\
  2q_b(k+1) &n=2k+1.
 \end{cases}
\end{gather*}
\end{prop}

By Lemmas~\ref{lem-two-quadrant-indecs} and~\ref{lem-two-quadrant-no-collisions}, every factor of length $n\geq 4$ uniquely defines a $\boxplus$-indecomposable permutation, and thus we have:

\begin{prop}\label{prop-bin-to-box}
Let $b$ be an infinite binary sequence, and let $w=\phi(b)$. The number of $\boxplus$-indecompo\-sable permutations in $\C_w$ of each length $n\geq 4$ is $p_b(k) + p_b(k+1)$ if $n=2k$, and $2p_b(k+1)$ if $n=2k+1$. Furthermore, if $w$ visits two quadrants infinitely often, then there are 2 $\boxplus$-indecomposable permutations of each length 1 and 2, and $2p_b(2)-2$ of length $3$.

Similarly, the number of $\boxplus$-indecomposable permutations in $\C_w^\boxplus$ of each length $n\geq 4$ is $q_b(k) + q_b(k+1)$ if $n=2k$, and $2q_b(k+1)$ if $n=2k+1$. Furthermore, if $w$ visits two quadrants infinitely often, then there are 2 $\boxplus$-indecomposable permutations of each length 1 and 2, and $2q_b(2)-2$ of length $3$.
\end{prop}

\subsection{A spectrum of growth rates}

Let $v= (\sfru\sfu\sfl\sfu)^\infty = \phi((10)^\infty)$. In the next two propositions, we show that $v$ is essentially the unique pin sequence in two quadrants for which $\C_v$ has the smallest possible growth rate.

\begin{prop}\label{prop-gr-nu}
$\gr(\C^\gridded_v) = \nu$, where $\nu\approx 3.06918$ is the largest real root of $z^4-3z^3-2$.	
\end{prop}

\begin{proof}
Since $v$ is periodic, $\C_{v}$ is $\boxplus$-closed by Proposition~\ref{prop-w-recurrent}. By inspection, there are 2 $\boxplus$-indecomposable permutations of each length $n=1$, 2 and 3 in $\C_v$. Since $v=\phi((01)^\infty)$ and $p_{(01)^\infty}(n)=2$ for all $n$, it follows by Proposition~\ref{prop-bin-to-box} that there are precisely 4 $\boxplus$-indecomposable permutations of each length $n\geq 4$ in $\C_v$. Thus the generating function for $\C_w$ is
%\[
%g(z) = 2z+2z^2+2z^3+\frac{4z^4}{1-z} = \frac{2z+2z^4}{1-z}
%\]
%and hence
\[f(z) = \frac1{1-\frac{2z+2z^4}{1-z}} = \frac{1-z}{1-3z-2z^4}\]
and the proposition follows.
\end{proof}

\begin{prop}[\cite{jarvis:pin-classes-i}]\label{prop-pins-two-quadrants}
Suppose that $w$ visits exactly two quadrants infinitely often. Then $\gr(\C^\gridded_w)\ge \nu$, with equality if and only if up to symmetry there exists a finite word $w'$ such that $w=w'(\sfl\sfu\sfr\sfu)^\infty$.
\end{prop}

\begin{proof}
By symmetry and since $\C_w=\C_{w'}$ for any infinite suffix $w'$ of $w$ (by Proposition~\ref{prop-truncate-gr}), it suffices to consider pin sequences $w$ that only visit quadrants 1 and 2. Thus $w$ is a word over the alphabet $\{\sfur,\sful,\sflu,\sfru\}$, and indeed we may also assume that $w=\phi(b)$ for some binary sequence $b$ (by removing the first letter of $w$ if necessary).

We consider the recurrent complexity of $b$, and therefore that of $w$. Each arrival of $w$ into a quadrant is witnessed in $b$ by an occurrence of $01$ or $10$. Thus $q_b(1)=2$ and by Theorem~\ref{thm-recurrent-morse-hedlund} therefore $q_b(n)\geq 2$ for all $n$. Proposition~\ref{prop-bin-to-box} now implies that the number of $\boxplus$-indecomposable permutations in $\C_w^\boxplus$ of lengths $n=1,2,3,4,\dots$ is at least $2,2,2,4,\dots$.  The proof of Proposition~\ref{prop-gr-nu} now shows therefore that $\gr(\C_w) = \gr(\C_w^\boxplus) \geq \nu$.

Now suppose that $b$ contains $00$ as a recurrent factor. Then $q_b(3)\geq 3$, and so the sequence of  $\boxplus$-indecomposable permutations in $\C_w^\boxplus$ of lengths $n=1,2,3,4,\dots$ is at least $2,2,4,6,\dots$. This sequence has generating function
\[\frac{2z+2z^3+2z^4}{1-z}\]
and hence
\[
\gr(\C_w^\boxplus)\geq \gr\left(\frac{1-z}{1-3z-2z^3-2z^4}\right) > 3.24.
\]
(Recall that the growth rate of a generating function refers to the growth rate of the sequence of coefficients of the formal power series.)

Similarly, if $b$ contains $11$ as a recurrent factor, then again we have $\gr(\C_w^\boxplus) > 3.24$.

Thus, if $\gr(\C_w)=\nu$, then $b$ cannot contain infinitely many copies of $00$ or $11$, and hence $b=b'(01)^\infty$ for some finite word $b'$. We now take $w'=\phi(b')$, so that $w=w' (\sfl\sfu\sfr\sfu)^\infty$, and the proof follows by Proposition~\ref{prop-truncate-gr} and Proposition~\ref{prop-gr-nu}.
\end{proof}

As well as showing that the pin sequence $v$ is essentially the unique pin sequence of smallest growth rate, the above proof in fact also shows that if either of the binary factors $00$ or $11$ appears as recurrent factors in a binary word $b$, then $\gr(\C_{\phi(b)}) > 3.24$. Our next task is to investigate the two-quadrant classes that contain one or both of these factors infinitely often. 

For positive integers $k$ and $\ell$, Let $b_{k,\ell}= (0^k1^\ell)^\infty$, and let $w_{k,\ell} = \phi(b_{k,\ell}) = ((\sfl\sfu)^k(\sfr\sfu)^\ell)^\infty$ denote the periodic pin sequence of period $2(k+\ell)$. Set $\nu_{k,\ell}=\gr(\C_{w_{k,\ell}})$, and in the case where $k=1$, we suppress the `1' and write simply $\nu_\ell$. See Figure~\ref{fig-w-2-1} for an example. 

\begin{figure}
\centering
\begin{tikzpicture}[scale=.25]
	%\plotpermgrid{14,17,12,15,8,11,6,9,2,5,1,3,0,7,4,13,10,19,16}
	\plotpermgrid{16,19,10,13,4,7,0,3,1,5,2,9,6,11,8,15,12,17,14}
	\draw (7.5,1) -- (1) -| (3) |- (2) -| (5) |- (4) -| (7) |- (6) -| (9) |- (8) -| (11) |- (10) -| (13) |- (12) -| (15) |- (14) -| (17) |- (16) -| (19); 
	\draw[very thick] (7,.5) -- ++(0,19);
\end{tikzpicture}\par 
	\caption{The start of the infinite permutation corresponding to the word $w_{1,2}$.}\label{fig-w-2-1}
\end{figure}

We begin with the cases in which a pin sequence visits both quadrants 1 and 2 infinitely often, and infinitely many of those visits in each quadrant contain at least three points.
%Clearly, $w_{k,\ell}$ contains $\sfr\sfu\sfr$ if and only if $k\geq 2$, and it contains $\sfl\sfu\sfl$ if and only if $\ell\geq 2$. Our next result, establishes that the case $k=\ell=2$ is the smallest growth rate where pin words can contain \emph{both} $\sfr\sfu\sfr$ and $\sfl\sfu\sfl$.

\begin{prop}\label{prop-w-2-2-gr}
Let $w$ be a pin sequence which visits only quadrants 1 and 2, and in which both $\sfr\sfu\sfr$ and $\sfl\sfu\sfl$ are recurrent factors. Then $\gr(\C_w) \geq \nu_{2,2} \approx 3.39752$, with equality if and only if $w$ is eventually equal to $w_{2,2}= \phi(b_{2,2})=(\sfl\sfu\sfl\sfu\sfr\sfu\sfr\sfu)^\infty $.
\end{prop}

\begin{proof}
First, $b_{2,2}=(0011)^\infty$ has complexity function $p_b(1)=2$ and $p_b(n)=4$ for $n\geq 2$. Thus $w_{2,2}$ has complexity function $p_w(1)=p_w(2)=2$, $p_w(3)=6$ and $p_w(n)=8$ for $n\geq 4$. This sequence has generating function $(2z+4z^3+2z^4)/(1-z)$. Since $w_{2,2}$ is periodic, it follows that 
\[
\nu_{2,2} = \gr(\C_{w_{2,2}}) = \gr\left( \frac{1-z}{1-3z-4z^3-2z^4}\right) \approx 3.39752.
\]
For reference $\nu_{2,2}$ is the largest real zero of $z^4-3z^3-4z-2$.

Now consider a pin sequence $w$ as in the statement of the proposition. As before, we may assume that $w=\phi(b)$ for some binary sequence $b$. 

Since $w$ contains both $\sfr\sfu\sfr$ and $\sfl\sfu\sfl$ as recurrent factors, $b$ must contain both $00$ and $11$ as recurrent factors. Furthermore, $b$ must also contain $01$ and $10$, and hence $p_b(2)\geq 4$. Therefore, the counting sequence of $\boxplus$-indecomposable permutations in $\C_w^\boxplus$ is at least $2,2,6,8,8,\dots$. This is the same sequence as the complexity of $w_{2,2}$, hence $\gr(\C_w)=\gr(\C_w^\boxplus) \geq \nu_{2,2}$. 

Now suppose that in addition, $b$ contains $000$ as a recurrent factor. In this case, $q_b(3)\geq 5$, and this implies that the counting sequence of $\boxplus$-indecomposable permutations in $\C_w^\boxplus$ is at least $2,2,6,9,10,10,\dots$. This has generating function $(2z+4z^3+3z^4+z^5)/(1-z)$ from which we find that $\gr(\C_w^\boxplus)\geq 3.423$.

Similar analysis applies if $b$ contains any of $111$, $101$ or $010$ as recurrent factors. Consequently, if $\gr(\C_w)=\nu_{2,2}$, then $b$ cannot contain any of $000$, $111$, $101$ or $010$ as recurrent factors, which implies that $b=b'(0011)^\infty$ for some finite word $b'$. The result now follows by the same arguments as used previously.
\end{proof}

We now investigate pin classes whose corresponding sequences visit quadrants 1 and 2 and whose growth rates are below $\nu_{2,2}$. We begin by considering the classes generated by the periodic sequences $w_{k,\ell}$. Since $\nu_{k,\ell} < \nu_{2,2}$, it follows from Proposition~\ref{prop-w-2-2-gr} that $k=1$ or $\ell=1$.

\begin{lemma}\label{lem-nu-1-ell}
We have $\nu_{2} = \nu_{1,2} \approx 3.24796$ (the largest real zero of $z^4-3z^3-2z-2$), while for $\ell>2$ the growth rate $\nu_{\ell}$ is equal to the largest real zero of
\[
z^{2\ell}-4z^{2\ell-1}+3z^{2\ell-2}-2z^{2\ell-3}-z^{2\ell-4}+2z^{2\ell-5}+1.
\]
\end{lemma}

\begin{proof}
By inspection, $b_{1,\ell} = (01^\ell)^\infty$ has complexity function
\[
p_{b_{1,\ell}}(n) =\begin{cases}
n+1&1\leq n\leq \ell\\
\ell+1&n> \ell.
\end{cases}
\]
Thus, the generating function for the $\boxplus$-indecomposable permutations in $\C_{w_{1,\ell}}$ is
\[
g_{1,\ell}(z)=\frac{2z+2z^3+3z^4+z^5+z^6+\cdots + z^{2\ell-1}}{1-z}  = \frac{2z-2z^2+2z^3+z^4-2z^5-z^{2\ell}}{(1-z)^2}
\]
if $\ell > 2$, and if $\ell=2$ then the generating function is 
\[
g_{1,2}(z) = \frac{2z+2z^3+2z^4}{1-z}.
\]
The result in the lemma now follows by standard combinatorial analysis.
\end{proof}

The following table gives the first few growth rates for these classes.

{\centering
\begin{tabular}{c|l|c}
	$\ell$& Sequence of box decomposables& $\nu_{\ell}=\gr(\C_{w_{1,\ell}})$\\\hline
	1& $2,2,2,4,\dots$&3.06918\\
	2& $2,2,4,6,6,\dots$&3.24796\\
	3& $2,2,4,7,8,8,\dots$&3.27963\\
	4& $2,2,4,7,8,9,10,10,\dots$&3.28248\\
	5& $2,2,4,7,8,9,10,11,12,12,\dots$&3.28274\\
	6& $2,2,4,7,8,9,10,11,12,13,14,14,\dots$&3.28277\\\hline
\end{tabular}\par}

%
%
%
%
%
%%%%%%%%%%%%%%%%%%
\section{Non-periodic classes in two quadrants}\label{sec-non-periodic}

As $\ell\to\infty$, the family of counting sequences in the table at the end of the previous section converges to the sequence $2,2,4,7,8,9,\dots$ with generating function
\[
g_\star(z) = \frac{2z-2z^2+2z^3+z^4-2z^5}{(1-z)^2}.
\]
The growth rate of $1/(1-g_\star(z))$ is $\mu \approx 3.28277$, the largest real zero of $z^5-4z^4+3z^3-2z^2-z+2$. It is this growth rate that features in our two main theorems (Theorems~\ref{thm-at-mu} and~\ref{thm-classification}), although note that a priori, we have not yet exhibited a single pin class with growth rate $\mu$. In fact, certainly there cannot exist a periodic pin sequence $w$ that generates a set of box-indecomposable permutations with generating function $g_\star(z)$, since the number of box-indecomposables generated by such a word must be eventually constant.

In this section we establish Theorem~\ref{thm-at-mu} by constructing uncountably many pin classes of growth rate $\mu$, each defined by a recurrent sequence. Additionally, we exhibit a pin class defined by a sequence that is neither recurrent nor eventually recurrent, also of growth rate $\mu$. In Section~\ref{sec-classification}, we go on to prove that the only possible growth rates of pin classes below $\mu$ are those in the table at the end of the preceding section.

A \emph{Sturmian sequence} is an infinite binary sequence $s$ that has complexity function $p_s(n)=n+1$ for all $n$. One example of such a sequence is the \emph{Fibonacci sequence},
\[
0100101001001\cdots
\]
which is most easily defined as the fixed point of the substitution in which $0\mapsto 01$ and $1\mapsto 0$.

Here, we state only a minimal number of properties of these sequences, and refer the reader to Chapter~6 of~\cite{fogg:substitutions:} for further information.

\begin{enumerate}[(1)]
	\item There are uncountably many Sturmian sequences with pairwise distinct sets of factors.  
	\item All Sturmian sequences are recurrent.
	\item Sturmian sequences are the minimal sequences that are not periodic, in the sense that any sequence for which there exists $n$ such that the sequence contains the same number of factors of length $n+1$ as of length $n$ is necessarily eventually periodic.
\end{enumerate}

The following result follows by Proposition~\ref{prop-bin-to-pin}.

\begin{prop}\label{prop-sturmian-to-pin}
For a Sturmian sequence $s$, $\phi(s)$ contains $n+3$ distinct factors of length $n\geq 1$.	
\end{prop}

We can now prove the following.

\begin{thm}\label{thm-at-mu}
There exist uncountably many distinct pin sequences $w$ with $\gr(\C_w)=\mu$.	
\end{thm}

\begin{proof}
Given a Sturmian sequence $s$, Proposition~\ref{prop-sturmian-to-pin} tells us that $\phi(s)$ contains precisely $n+3$ distinct factors of each length $n$. By considering small lengths separately, and noting that each pin factor of length $4$ or more gives rise to a distinct box-indecomposable permutation (by Lemmas~\ref{lem-two-quadrant-indecs} and~\ref{lem-two-quadrant-no-collisions}), we see that the sequence of box indecomposables in $\C_{\phi(s)}$ is $2,2,4,7,8,9,\dots$. Thus the box indecomposables in $\C_{\phi(s)}$ have generating function $g_\star(z)$.

Since the Sturmian sequence $s$ is recurrent, it follows that $\C_{\phi(s)}$ is $\boxplus$-closed, and thus the generating function for $\C_s$ is $1/(1-g_\star(z))$, and hence $\gr(\C_{\phi(s)}) = \mu$.

Finally, given two distinct Sturmian sequences $s$ and $t$, the fact that $s$ and $t$ have distinct sets of factors means that $\phi(s)$ and $\phi(t)$ also have distinct sets of factors, and thus by Lemma~\ref{lem-two-quadrant-no-collisions} $\C_{\phi(s)}$ and $\C_{\phi(t)}$ have distinct sets of box-indecomposable permutations. The result now follows by property (1) above.	
\end{proof}

Our second task in this section is to exhibit one further pin class with growth rate $\mu$, which is \emph{not} box-closed. Define $b^{(1)} = 10$, and for $i\geq 2$ iteratively define
\[
b^{(i)} = b^{(i-1)}1^i0.
\]
Define $b^\star = \lim_{i\to\infty}b^{(i)} = 10\,110\,1110\,11110\,111110\,\cdots$, and let $w^\star = \phi(b^\star)$. Note that $b^\star$ (and hence also $w^\star$) is not periodic or recurrent, since any factor of the form $01^i0$ appears precisely once in $b^\star$. Consequently, the class $\C_{w^\star}$ is not box-closed, but we can appeal to Theorem~\ref{thm-jarvis-box-interior}
%, so some additional work is required in order to derive its growth rate. One direction of the following proof is based on a method mentioned in the concluding remarks of Vatter~\cite{vatter:growth-rates-of:}.

\begin{prop}\label{prop-liouville-V}
$\gr(\C_{w^\star}) = \mu$.
\end{prop}

\begin{proof}
We claim that 
\[ \C_{w^\star}^\boxplus = \bigboxplus \left\{ \pi^\gridded_w : w \text{ is a factor of }\phi(1^k01^\ell)=(\sfr\sfu)^k\sfl\sfu(\sfr\sfu)^\ell\text{ for some }k,\ell\right\}.\]
That the left hand side contains the right is clear, since every sequence of the form $1^k01^\ell$ is a recurrent factor of $b^\star$. Conversely, note that any factor of $w^\star$ which contains two or more left steps appears precisely once in $w^\star$. Thus, the only factors of $w^\star$ that are recurrent are those that contain at most one left step, and these are precisely the factors that appear in the right hand side of the claimed equality.

%To see this, consider a permutation $\pi^\gridded = \pi_1\boxplus \cdots \boxplus\pi_m\in\DDD$ such that $\pi_i$ is box-indecomposable for $1\leq i\leq m$. Choose $k$ and $\ell$ sufficiently large so that, with $w=(\sfr\sfu)^k\sfl\sfu(\sfr\sfu)^\ell$, we have $\pi_i \leq \pi_w$ for all $i$. Now $\pi_w \boxplus \cdots \boxplus \pi_w$ belongs to $\C_{w^\star}$ since $w^\star$ contains arbitrarily many disjoint copies of $(\sfr\sfu)^k\sfl\sfu(\sfr\sfu)^\ell$. Thus $\pi\in\C_{w^\star}$ and the proof of the claim is complete.

It is now straightforward to check that $\C_{w^\star}^\boxplus$ contains $n+3$ box-indecomposable permutations of each length $n\geq 4$, while for $n=1,2,3$ there are $2,2,4$ box-indecomposable permutations, respectively. By the earlier calculations, this tells us that $\gr(\C_{w^\star}^\boxplus)=\mu$, and hence $\gr(\C_{w^\star})= \mu$ by Theorem~\ref{thm-jarvis-box-interior}.
\end{proof}

%
%
%
%
%
%
%%%%%%%%%%%%%%%%%%%
\section{Classification of pin classes below \texorpdfstring{$\mu$}{mu}}\label{sec-classification}

We now complete a full classification of the pin classes whose growth rates are less than $\mu$. Such pin classes must be defined by sequences that visit at most two quadrants infinitely often (by Propositions~\ref{prop-4-quadrants} and~\ref{prop-3-quadrants}), and Proposition~\ref{prop-kill-finite-visits} allows us to restrict our attention to pin sequences that visit at most two quadrants. Consequently, we can now proceed by studying infinite binary sequences, and utilise the injection~$\phi$ described in Subsection~\ref{subsec-bin-to-pin-words}.

\begin{lemma}\label{lem-sub-mu-periodic}
For any infinite binary sequence $b$, if $\gr(\C_{\phi(b)})<\mu$, then $b$ is eventually periodic.	
\end{lemma}

\begin{proof}
Since $\gr(\C_{\phi(b)})<\mu$ and $\mu$ is the growth rate of the pin class associated with Sturmian sequences, the recurrent complexity function $q_b$ must satisfy, for some $n$, $q_b(n) < n+1$. 

Pick the smallest such $n$, so that $q_b(n-1) \geq n$. Since $q_b(n)\geq q_b(n-1)$ (by the first part of Theorem~\ref{thm-recurrent-morse-hedlund}), it follows that $q_b(n) = q_b(n-1)=n$. Now, by the second part of Theorem~\ref{thm-recurrent-morse-hedlund} we see that $b$ is eventually periodic.
\end{proof}

With this established, we now narrow the range of potential binary sequences that give rise to classes with growth rate below $\mu$.

\begin{prop}\label{prop-not-above-sturmian}
Let $b$ be a binary sequence for which $q_b(n) > n+1$ for some $n$. Then $\gr(\C_{\phi(w)}) \geq \mu$.
\end{prop}

\begin{proof}
First, recall that the sequence of $\boxplus$-indecomposable permutations generated by a Sturmian sequence via the mapping $\phi$ has the generating function
\[
g_\star(z) = \frac{2z-2z^2+2z^3+z^4-2z^5}{(1-z)^2}.
\]
Note that $g_\star(1/\mu) = 1$, and $z=1/\mu \approx 0.30462$ is the smallest positive real solution of $g_\star(z)=1$. %Furthermore, for $z\in (0,1/\mu)$, $g_\star(z)$ is strictly increasing and has range $(0,1)$.

For each $k>1$, set
\[s_k(n) = \begin{cases}
 n+1 & n < k\\
 k+2 & n \geq k.
\end{cases}
\]
Consider any binary sequence $b$ which satisfies the hypothesis of the proposition. Note that $b$ must have recurrent complexity satisfying $q_b(n) \geq s_k(n)$ for some $k>1$ and for all $n$. (The case $k=1$ can be excluded, since $q(1) \leq 2$ for any binary sequence.)

Suppose first that $k\geq 4$ (we will handle small cases separately later). Let $g_b(z)$ denote the generating function for the non-empty $\boxplus$-indecomposable permutations in $\C_{\phi(b)}^\boxplus$. By applying Proposition~\ref{prop-bin-to-box} to the sequence $s_k(n)$, we see that
\[
[z^n]g_b(z) \geq \begin{cases}
 	2&n\leq 2\\
 	4&n=3\\
 	n+3& 4\leq n \leq 2k-3\\
 	2k+2&n=2k-2\\
 	2k+4&n\geq 2k-1.
\end{cases}
\]
The generating function for the coefficients on the right hand side of the above expression is
\[
g_s(z)=\frac{2z-2z^2+2z^3+z^4-2z^5 + z^{2k-2}-2z^{2k} }{(1-z)^2}.
\]
Importantly, we have that
\[
g_s(z) - g_\star(z) = \frac{z^{2k-2}-2z^{2k}}{(1-z)^2} = z^{2k-2}\frac{1-2z^2}{(1-z)^2} \geq 0 \text{ for }0\le z\le\tfrac12.
\]
Thus for $0\le z\le\tfrac12$ we see that $g_b(z) \geq g_s(z) \geq g_\star(z)$. Now $z=1/\mu\approx 0.30562$ is the smallest positive real solution to $g_\star(z)=1$, and since $g_b(z) \geq g_\star(z)$ for $0\le z\le\tfrac12$ it follows that the smallest positive real solution to $g_b(z)=1$ is at most $1/\mu$. This value of $z$ is the reciprocal of $\gr\left(\C_{\phi(b)}^\boxplus\right)$, from which we see that
\[
\gr(\C_{\phi(b)}) = \gr\left(\C_{\phi(b)}^\boxplus\right) \geq \mu.
\]
We now briefly cover the cases $k=2$ and $k=3$. When $k=2$, the sequence of $\boxplus$-indecomposable permutations in $\C_{\phi(b)}^\boxplus$ dominates the sequence $2,2,6,8,8,8,\dots$, which has generating function $(2z+4z^3+2z^4)/(1-z)$. The solution to $(2z+4z^3+2z^4)/(1-z)=1$ is $z\approx 0.29433$, and therefore $\gr\left(\C_{\phi(b)}^\boxplus\right) \geq 1/0.29434 \approx 3.397 > \mu$.

Similarly, when $k=3$, the sequence of $\boxplus$-indecomposables is at least $2,2,4,8,10,10,\dots$, and this yields a class whose growth rate is at least $3.310 > \mu$.
\end{proof}

The following theorem now completes the classification of small pin sequences.

\begin{thm}\label{thm-classification}
Let $w$ be an infinite pin sequence such that $\gr(\C_w) < \mu$. Then $w$ is eventually periodic, and either $\gr(C_w) =\kappa$, or $\gr(C_w) = \nu_{\ell}$ for some $\ell \geq 1$.
\end{thm}

\begin{proof}
By the remarks at the beginning of this section, we can restrict our attention to pin sequences $w$ that visit at most two quadrants, and thus $w=\phi(b)$ for some binary word $b$. Lemma~\ref{lem-sub-mu-periodic} shows that $b$ is eventually periodic, and note that $b$ is eventually periodic if and only if $w$ is eventually periodic. 

To establish the growth rate of $\C_w$, note that by Proposition~\ref{prop-truncate-gr} we can assume that $w$ is in fact periodic. Thus $p_b(n)=q_b(n)$ for all $n$, and the resulting class $\C_{\phi(b)}$ is $\boxplus$-closed. 
Proposition~\ref{prop-not-above-sturmian} shows that a binary sequence $b$ for which $\gr(\C_{\phi(b)})<\mu$ has a complexity function that satisfies $p_b(n) \leq n+1$ for all $n$, and Lemma~\ref{lem-constant-forever} now implies that there exists $\ell\geq 2$ such that 
\[
p_b(n) = \begin{cases}
 n+1 & n<\ell\\
 \ell+1 &n\geq \ell.	
 \end{cases}
\]
The statement and proof of Lemma~\ref{lem-nu-1-ell} now complete the proof: the generating function for the $\boxplus$-indecomposable permutations in $\C_{\phi(b)}=\C_{\phi(b)}^\boxplus$ is $g_{\phi(b)} = g_{1,\ell}(z)$, and consequently $\gr(\C_{\phi(b)}) =\nu_{\ell}$.
\end{proof}

%
%
%
%
%
%
%
%
%
%

%
%
%
%
%
%
%
%
%%%%%%%%%%%%%%%%%%%%%%
%\section{A section}

%
%
%
%
%
%
%
%
%%%%%%%%%%%%%%%%%%%%%%
\section{Concluding remarks}\label{sec-conclusion}

At first sight, the phase transition we have established at $\mu$ might
 appear to bear some resemblance to the one at $\kappa$ (as observed by~\cite{vatter:small-permutati:}). For example, $\kappa\approx 2.20557$ is the first growth rate of uncountably many classes, while $\mu$ is the first growth rate of uncountably many \emph{pin} classes. However, they do also display significant differences; note that for any $\epsilon >0$, there are only finitely many growth rates of pin classes less than $\mu-\epsilon$, yet the first accumulation point of permutation classes occurs at growth rate 2 (which is the limit of the growth rates of permutation classes enumerated by the generalised Fibonacci numbers), and in fact $\kappa$ is an accumulation point of accumulation points. 
 
 In fact, $\mu$ has one further curious property: it is also an accumulation point of growth rates \emph{from above} as the following family of pin classes shows. For $k\geq 2$, consider the pin sequence $w_k = \left((\sfr\sfu)^k\sfl\sfu(\sfr\sfu)^{2k}\right)^\infty$. This generates a pin class whose sequence of $\boxplus$-indecomposable permutations has generating function
 \[
 g_k(z) = \frac{2z-2z^2+2z^3+z^4-2z^5+z^{2k+2}-z^{4k}-z^{4k+4}}{(1-z)^2}
 \]
 It is clear that the series expansion of $g_k(z)$ approaches that of $g_\star(z)$ as $k\to\infty$, and straightforward to verify that the solutions to $g_k(z)=1$ are smaller than $1/\mu$, and monotone increasing with $k$. Thus, the growth rate of the class $\C_{w_k}$ exceeds $\mu$, but approaches $\mu$ as $k\to\infty$.
 
 \acknowledgements We are grateful to the referees' careful reading of an earlier draft, which improved the readability of the paper.
 
%One might also wonder what the \emph{largest} growth rate of a pin class is. One can construct an infinite pin sequence that contains every finite pin word simply by listing all the finite pin words in some order, and concatenating these (possibly with additional pins to ensure the resulting word is a valid pin sequence). Thus, there exists a pin class that contains every pin sequence, so we can appeal to the enumeration of pin permutations by Bassino, Bouvel and Rossin~\cite{bassino:enumeration-of-:} (sequence A138619 of the OEIS~\cite{oeis}), who found the generating function for the class of all \emph{ungridded} pin permutations to be
%\[
%f(z) = \frac{1-7z+13z^2-17z^3+2z^4-8z^5+12z^6+12z^7-8z^8}{1-8z+19z^2-26z^3+14z^4-12z^5-8z^6+20z^7-8z^8}.
%\]
%The reciprocal of the smallest real root of the denominator gives us a growth rate of approximately $5.24112$.

% ------------------------------------------------------------------------------------------------
%\bibliographystyle{abbrvnat}
%\bibliography{refs}
%\urlstyle{rm}
\def\cprime{$'$}

\end{document}

%% file: tikzperm.tex
\usepackage{ifthen}
\usetikzlibrary{calc}
\usetikzlibrary{decorations.pathreplacing,decorations.markings}

\newcommand{\plotptradius}{4pt}
\newcommand{\setplotptradius}[1]{\renewcommand{\plotptradius}{#1}}
\tikzset{permpt/.style={circle, draw, fill=black, inner sep=0pt, minimum width=\plotptradius}}
\tikzset{empty/.style={draw=none, fill=none}}

\newcommand\absdot[2]{
	\node[permpt] at #1 {};
}

\newcommand{\plotperm}[2][black]{ %[colour]{permutation}
	\foreach \j [count=\i] in {#2} {
        \ifnum0=\j {} \else {
 		\node[permpt,fill=#1,draw=#1] (\j) at (\i,\j) {};
	} \fi
	}
}
\newcommand{\plotpermbox}[4]{
	\draw [darkgray, very thick, line cap=round, fill=white]
		({#1-0.5}, {#2-0.5}) rectangle ({#3+0.5}, {#4+0.5});
}

\newcommand{\plotpermborder}[1]{
	\foreach \i [count=\nn] in {#1} {\global\let\n\nn}    % scope of \nn now restricted to loop, resolving tikz issue #702, 3 Jul 2019
	\plotpermbox{1}{1}{\n}{\n}
	\plotperm{#1}
}
\newcommand{\plotpermbordergrid}[1]{
	\foreach \i [count=\nn] in {#1} {\global\let\n\nn}    % scope of \nn now restricted to loop, resolving tikz issue #702, 3 Jul 2019
	\plotpermbox{1}{1}{\n}{\n}
	\draw[step=1cm,gray!50,very thin] (1,1) grid (\n,\n);
	\plotperm{#1}
}
\newcommand{\plotgrid}[1]{
	\draw[step=1cm,gray!50,very thin] (0.5,0.5) grid (#1+0.5,#1+0.5);
}

\newcommand{\plotpermgrid}[1]{
	\foreach \i [count=\nn] in {#1} {\global\let\n\nn}    % scope of \nn now restricted to loop, resolving tikz issue #702, 3 Jul 2019
	\plotgrid{\n}
	\plotperm{#1}
}

\newcommand{\plotpinsequence}[1]{
	\absdot{(0,0)}{};
	\edef\n{0}
	\edef\s{0}
	\edef\e{0}
	\edef\w{0}
	\edef\x{0}
	\edef\y{0}
	\foreach \pin [remember=\pin as \oldpin (initially 1), count=\i] in {#1} {
		\ifthenelse{\pin=1 \OR \pin=2}{%up or down
			\ifthenelse{\oldpin=3}{% previous=right
				\xdef\x{\number\numexpr\e-1}
			}{
				\xdef\x{\number\numexpr\w+1}
			}
			\ifnum\i=1 %expand eastern box by 1 if 1st pin
				\pgfmathparse{\e+1}
 				\xdef\e{\pgfmathresult}
			\fi	
		}{ %left or right
			\ifthenelse{\oldpin=1}{% previous=up
				\xdef\y{\number\numexpr\n-1}
			}{
				\xdef\y{\number\numexpr\s+1}
			}
			\ifnum\i=1 %expand southern boundary by 1 if 1st pin
				\pgfmathparse{\s-1}
 				\xdef\s{\pgfmathresult}
			\fi	
		}
		\ifnum\pin=1 %up
			\pgfmathparse{\n+2}
 			\xdef\n{\pgfmathresult}		
			\absdot{(\x,\n)}{};
			\ifnum\i>1
				\draw (\x,\n) -- (\x,\y-0.5);
			\else
				\draw[gray,very thick] (-0.5,-0.5) rectangle (\x+0.5,\n+0.5);
			\fi
		\fi
		\ifnum\pin=2 % down		
			\pgfmathparse{\s-2}
 			\xdef\s{\pgfmathresult}
			\absdot{(\x,\s)}{};
			\ifnum\i>1
				\draw (\x,\s) -- (\x,\y+0.5);
			\else
				\draw[gray,very thick] (-0.5,0.5) rectangle (\x+0.5,\s-0.5);
			\fi
		\fi
		\ifnum\pin=3 %right
			\pgfmathparse{\e+2}
 			\xdef\e{\pgfmathresult}
			\absdot{(\e,\y)}{};
			\ifnum\i>1
				\draw (\e,\y) -- (\x-0.5,\y);
			\else
				\draw[gray,very thick] (-0.5,+0.5) rectangle (\e+0.5,\y-0.5);
			\fi
		\fi
		\ifnum\pin=4 %left
			\pgfmathparse{\w-2}
 			\xdef\w{\pgfmathresult}
			\absdot{(\w,\y)}{};
			\ifnum\i>1
				\draw (\w,\y) -- (\x+0.5,\y);
			\else
				\draw[gray,very thick] (0.5,0.5) rectangle (\w-0.5,\y-0.5);

			\fi
		\fi		
	};
}

\tikzset{
  on each segment/.style={
    decorate,
    decoration={
      show path construction,
      moveto code={},
      lineto code={
        \path [#1]
        (\tikzinputsegmentfirst) -- (\tikzinputsegmentlast);
      },
      curveto code={
        \path [#1] (\tikzinputsegmentfirst)
        .. controls
        (\tikzinputsegmentsupporta) and (\tikzinputsegmentsupportb)
        ..
        (\tikzinputsegmentlast);
      },
      closepath code={
        \path [#1]
        (\tikzinputsegmentfirst) -- (\tikzinputsegmentlast);
      },
    },
  },
  mid arrow/.style={postaction={decorate,decoration={
        markings,
        mark=at position .5 with {\arrow[#1]{stealth}}
      }}},
}